\documentclass[12pt]{amsart}

\usepackage{amsmath}
\usepackage{amsthm}
\usepackage{amssymb}
\usepackage{amsfonts}
\usepackage{bm}
\usepackage[shortlabels]{enumitem}
\usepackage[bookmarks=true,hyperindex,pdftex,colorlinks,citecolor=red, linkcolor=blue]{hyperref}
\usepackage{centernot}
\usepackage{marginnote}
\usepackage[left=2.5cm,right=2.5cm,top=2cm,bottom=2cm]{geometry}
\usepackage{bbm}
\usepackage{xcolor}

\theoremstyle{plain}
\newtheorem{thm}{Theorem}[section]

\newtheorem{cor}[thm]{Corollary}
\newtheorem{remark}[thm]{Remark}

\theoremstyle{definition}

\newtheorem{q}[thm]{Question}

\newcommand{\R}{\mathbb{R}}
\newcommand{\norm}[1]{\left\Vert #1\right\Vert}

\begin{document}
	
	\title{Stability of Individually Eventually Positive Semigroups on $L^p$-Spaces}
	\author[L. Arnold]{Loris Arnold}

	\address[L. Arnold]{Normandie Univ, UNICAEN, CNRS, LMNO, 14000 Caen, France}
	\email{lfj.arld@gmail.com}
	
	
	\maketitle
	
	\begin{abstract}
		We answer a question raised by Vogt by proving that the growth bound of an individually eventually positive $C_0$-semigroup on an $L^p$-space ($1 < p < \infty$) coincides with the spectral bound of its generator. Our proof follows the general strategy of Vogt's argument for the uniformly eventually positive case. The key new ingredient is a variant of an operator-range uniformisation
		principle of Arora and Gl\"uck, adapted here to time-dependent operator ranges. When applied to principal ideals, this principle turns individual eventual positivity into the uniform domination estimate required for Vogt's stability argument.
	\end{abstract}
	
	\section{Introduction}
	
	Let $(T_t)_{t\ge 0}$ be a $C_0$-semigroup on a complex Banach lattice $E$
	with generator $A$. The semigroup is said to be \emph{individually
		eventually positive} if for every $f\in E_+$ there exists $t_f\ge 0$ such
	that
	\[
	T_tf\ge 0 \qquad (t\ge t_f).
	\]
	This property is substantially weaker than \emph{uniform eventual
		positivity}, where $t_f$ can be chosen independently of $f$.

	Lutz Weis proved that $\omega_0(T)=s(A)$ for positive $C_0$-semigroups on $L^p$ in \cite{Weis95} (see also \cite{Weis98}), while Vogt \cite{Vogt} extended this result to uniformly eventually positive $C_0$-semigroups on $L^p$-spaces via a short duality argument.
	
	The main difficulty in extending this argument to individually eventually
	positive semigroups is the lack of a common positivity time. Although each
	positive orbit eventually becomes positive, the corresponding time may
	depend on the initial vector. The difficulty is that Vogt's duality argument needs a uniform domination estimate, whereas individual eventual positivity only provides pointwise eventual information.
	
	The purpose of this note is to show that this lack of uniformity can be
	overcome.
	
	\begin{thm}\label{thmMain}
		Let $1<p<\infty$ and let $(T_t)_{t\ge 0}$ be an individually eventually
		positive $C_0$-semigroup on $L^p(\Omega,\mu)$ with generator $A$. Then
		\[
		\omega_0(T)=s(A).
		\]
	\end{thm}
	
	The key observation is that one does not need to recover uniform eventual
	positivity of the semigroup on the whole space. It is enough to uniformise
	a suitable eventual order estimate on an operator range. Using a variant of
	the operator-range principle of Arora and Gl\"uck
	\cite[Theorem 2.6]{AroraGlueck2024}, we show that, for suitable bounded real
	linear operators
	\[
	J\colon Y\longrightarrow E_{\mathbb R}
	\]
	whose range is contained in a principal ideal $E_g$, a pointwise eventual
	estimate of the form
	\[
	|T_tJy|\lesssim \norm{y}_Y T_tg
	\]
	can be made uniform in $y$. There exist $t_0\ge 0$ and $C\ge 0$ such that
	\begin{equation}\label{eqDom}
		|T_tJy|
		\le
		C\norm{y}_Y T_tg
		\qquad
		(t\ge t_0,\ y\in Y).
	\end{equation}
	Applying this principle to local time averages of a single orbit yields the
	uniform domination estimate needed to carry out Vogt's $L^p$-duality
	argument and then apply Datko's theorem.
	
	\subsection*{Notation and terminology:}
	For Banach spaces $X$ and $Y$, we denote by
	\[
	B(X,Y)
	\]
	the Banach space of bounded linear operators from $X$ to $Y$, endowed with
	the operator norm, and we write
	\[
	B(X):=B(X,X).
	\] If $E$ is a complex Banach lattice, we denote by $E_{\mathbb R}$
	its real part and by $E_+$ its positive cone. For $f\in E_{\mathbb R}$,
	we write $f\geq0$ if $f\in E_+$, and $|f|$ denotes the lattice modulus of
	$f$. The same notation $|\cdot|$ is used for the modulus in the
	complexification of $E$.
	
	For $g\in E_+$, the principal ideal generated by $g$ is
	\[
	E_g
	:=
	\{f\in E:\ |f|\leq c g
	\text{ for some }c\geq0\}.
	\]
	Equipped with the gauge norm
	\[
	\norm{f}_g
	:=
	\inf\{c\geq0:\ |f|\leq c g\},
	\qquad f\in E_g,
	\]
	the space $E_g$ is itself a Banach lattice, and the canonical embedding
	$E_g\hookrightarrow E$ is continuous.
	
	Moreover, $(\Omega,\Sigma,\mu)$ denotes a measure space.
	For $1\leq p<\infty$, the space $L^p(\Omega,\mu)$ is understood to consist
	of equivalence classes of complex-valued measurable functions, while $
	L^p(\Omega,\mu;\mathbb R)$ denotes its real part. Its positive cone is
	\[
	L^p(\Omega,\mu)_+
	=
	\{f\in L^p(\Omega,\mu;\mathbb R): f\geq0 \text{ a.e.}\}.
	\]
	As usual, inequalities between elements of $L^p(\Omega,\mu)$ are understood
	in the almost everywhere sense.
	
	If $(T_t)_{t\geq0}$ is a $C_0$-semigroup with generator $A$, we denote by
	\[
	s(A)
	:=
	\sup\{\operatorname{Re}\lambda:\lambda\in\sigma(A)\}
	\]
	the spectral bound of $A$, and by
	\[
	\omega_0(T)
	:=
	\inf\left\{
	\omega\in\mathbb R:
	\text{there exists }M\geq1\text{ such that }
	\norm{T_t}\leq Me^{\omega t}
	\text{ for all }t\geq0
	\right\}
	\]
	the growth bound of the semigroup.
	
	\subsection*{Organization of the paper}The paper is organised as follows. In Section 2, we prove the
	uniformisation theorem for varying operator ranges and apply it to principal
	ideals. In Section 3, we prove the main theorem and conclude with a remark presenting a new proof for AM-spaces.
	
	\section{A uniformisation principle for time-dependent operator ranges}
	
	We first present the operator-range principle that will be used below; an analogous statement can be found in \cite[Theorem 2.6]{AroraGlueck2024}.
	
	\begin{thm}
		\label{thmRange}
		Let $E$ and $F$ be Banach spaces, and let $(S_i)_{i\in I}$ be a net of bounded linear operators from $E$ to $F$. Assume that the directed set $I$ contains a countable majorizing subset. For each $i\in I$, let $W_i\subseteq F$ be a vector subspace endowed with a complete norm $\norm{\cdot}_{W_i}$ such that the inclusion $W_i\hookrightarrow F$ is continuous (that is, $W_i$ is an operator range in $F$). Then the following statements are equivalent.
		\begin{enumerate}[label=\upshape(\alph*)]
			\item\label{itUnif}
			There exist $M_0\ge 0$ and $i_0\in I$ such that, for every $i\ge i_0$,
			\[
			S_i E \subseteq W_i
			\qquad\text{and}\qquad
			\norm{S_i}_{B(E, W_i)} \le M_0.
			\]
			
			\item\label{itIndiv}
			There exists $M_1\ge 0$ such that for every $x\in E$ there exists $i_x\in I$ for which
			\[
			S_i x \in W_i
			\qquad\text{and}\qquad
			\norm{S_i x}_{W_i} \le M_1\norm{x}_E
			\]
			for every $i\ge i_x$.
		\end{enumerate}
	\end{thm}
	
	\begin{proof}
		The proof is essentially the same as that of \cite[Theorem 2.6]{AroraGlueck2024}. For the sake of completeness, we give a short proof. The implication \ref{itUnif}$\Rightarrow$\ref{itIndiv} is immediate.
		
		Conversely, let $I_0\subseteq I$ be a countable majorizing subset. For $i\in I_0$, set
		\[
		V_i := \left\{ x\in E \mid S_j x\in W_j \text{ for all } j\ge i \text{ and } \sup_{j\ge i}\norm{S_j x}_{W_j}<\infty \right\}
		\]
		and endow $V_i$ with the norm
		\[
		\norm{x}_{V_i} := \norm{x}_E + \sup_{j\ge i}\norm{S_j x}_{W_j}.
		\]
		
		A standard completeness argument shows that $V_i$ is an operator range in $E$. Indeed, if $(x_n)$ is Cauchy in $V_i$, then it converges in $E$, while for each $j\ge i$ the sequence $(S_j x_n)$ converges in $W_j$. The continuity of the embedding $W_j\hookrightarrow F$ and the boundedness of $S_j\colon E\to F$ identify this limit with $S_j x$.
		
		Assumption \ref{itIndiv} and the majorizing property of $I_0$ show that
		\[
		E = \bigcup_{i\in I_0} V_i.
		\]
		By \cite[Proposition 2.6]{AroraGlueck2023}, there exists $i_0\in I_0$ such that 
		\[
		V_{i_0} = E.
		\]
		The open mapping theorem then yields a constant $C\ge 0$ such that for every $x\in E$
		\[
		\sup_{j\ge i_0}\norm{S_j x}_{W_j} \le C\norm{x}_E.
		\]
		This is precisely \ref{itUnif}.
	\end{proof}
	
	For $g\in E_+$, recall that the principal ideal generated by $g$ is
	\[
	E_g := \{x\in E \mid |x|\le c g \text{ for some } c\ge 0\},
	\]
	endowed with the gauge norm
	\[
	\norm{x}_g := \inf\{c\ge 0 \mid |x|\le c g\}.
	\]
	Then $(E_g,\norm{\cdot}_g)$ is a complex Banach lattice and the inclusion $E_g\hookrightarrow E$ is continuous.
	
	\begin{cor}\label{corDom}
		Let $(T_t)_{t\ge 0}$ be an individually eventually positive $C_0$-semigroup on a Banach lattice $E$. Let $g\in E_+$, let $Y$ be a real Banach space, and let $J\colon Y\longrightarrow E_{\R}$ be a bounded linear operator such that $JY\subseteq E_g$. Then there exist $t_0\ge 0$ and $C\ge 0$ such that for all $t \ge t_0$ and $y\in Y$,
		\begin{equation}
			\label{eqUdom}
			|T_t Jy| \le C\norm{y}_Y T_t g.
		\end{equation}
		In particular, $T_t g \ge 0$ for all $t\ge t_0$.
	\end{cor}
	
	\begin{proof}
		By the closed graph theorem, $J\colon Y \longrightarrow (E_g)_{\R}$ is automatically bounded. We set $M := \norm{J}_{B(Y, E_g)}$ and choose $t_g\ge 0$ such that $T_t g \ge 0$ for all $t\ge t_g$.
		
		For $t\ge t_g$, we define the operator range $W_t := E_{T_t g}$ equipped with its gauge norm. Fix $y\in Y$. Since $|Jy|\le M\norm{y}_Y g$, the two vectors
		\[
		M\norm{y}_Y g + Jy \qquad \text{and} \qquad M\norm{y}_Y g - Jy
		\]
		belong to $E_+$. By individual eventual positivity, there exists $t_y \ge t_g$ such that $T_t(M\norm{y}_Y g + Jy)$ and $T_t(M\norm{y}_Y g - Jy)$ are positive for every $t\ge t_y$. Hence, for every $t \ge t_y$,
		\[
		-M\norm{y}_Y T_t g \le T_t Jy \le M\norm{y}_Y T_t g.
		\]
		That is, $|T_t Jy| \le M\norm{y}_Y T_t g.$ Equivalently, for every $t \ge t_y$,
		\[
		T_t Jy \in W_t \qquad \text{and} \qquad \norm{T_t Jy}_{W_t} \le M\norm{y}_Y.
		\]
		
		We may now apply Theorem \ref{thmRange} to the family of operators $(T_t J)_{t\ge t_g}$ from $Y$ to $E$ and the family of operator ranges $(W_t)_{t\ge t_g}$ in $E$. Since $[t_g,\infty)$ contains a countable majorizing subset, the desired uniform time $t_0$ and constant $C$ follow.
	\end{proof}
	
	\section{Proof of the main theorem}
	
	We now apply the preceding domination principle to Vogt's argument.
	
	\begin{proof}[Proof of Theorem \ref{thmMain}]
		
		Since $s(A) \le \omega_0(T)$ holds universally, applying a shift $e^{-\lambda t}T_t$ for $\lambda > s(A)$ reduces the claim to proving that
		\[
		s(A) < 0 \implies \omega_0(T) < 0.
		\]
		Assume $s(A) < 0$ and fix $f \in L^p(\Omega)_+$. By individual eventual positivity, $T_t f \ge 0$ for all $t \ge a$ for some $a \ge 0$. Setting $p' = \frac{p}{p-1}$ and $Y := L^{p'}([a,a+1]; \mathbb{R})$, we define $J \colon Y \to L^p(\Omega; \mathbb{R})$ by
		\[
		Jh := \int_a^{a+1} h(u) T_u f \, du.
		\]
		The function $F := \left( \int_a^{a+1} |T_u f|^p \, du \right)^{1/p}$ belongs to $L^p(\Omega)_+$ since, by Tonelli's theorem, $\norm{F}_p^p = \int_a^{a+1} \norm{T_u f}_p^p \, du < \infty$.
		Moreover, H\"older's inequality gives, for every $h\in Y$,
		\begin{equation}
			\label{eqJf}
			|Jh|
			\le
			\norm{h}_{p'}F.
		\end{equation}
		Thus
		\[
		JY\subseteq E_F.
		\]
		
		By Corollary \ref{corDom}, there exist $\tau \ge a$ and $C \ge 0$ such that, for all $u \ge \tau$ and $h \in Y$,
		\begin{equation}
			\label{eqMdom}
			|T_u Jh| \le C \norm{h}_{p'} T_u F \qquad \text{and} \qquad T_u F \ge 0.
		\end{equation}
		Set $b := a + \tau + 1$. Fix $R > b$ and $0 \le h \in L^{p'}([b,R])$ with $\norm{h}_{p'} \le 1$, extended by zero to $\mathbb{R}$. For $u \ge \tau$, the function $h_u [a,a+1] \rightarrow \R$ defined by $h_u(r) := h(u+r)$ satisfies $h_u \in Y$ and $\norm{h_u}_{p'} \le 1$.
		By Fubini's theorem and the semigroup property,
		\begin{align}
			\int_b^R h(t)T_tf\,dt
			&=
			\int_b^R
			\left(
			\int_{t-a-1}^{t-a}du
			\right)
			h(t)T_tf\,dt
			\notag\\
			&=
			\int_\tau^{R-a}
			\int_{a+u}^{a+u+1}
			h(t)T_tf\,dt\,du
			\notag\\
			&=
			\int_\tau^{R-a}
			T_u
			\left(
			\int_a^{a+1}
			h(u+r)T_rf\,dr
			\right)
			du
			\notag\\
			&=
			\int_\tau^{\infty}
			T_uJ(h_u)\,du.
			\label{eqFub}
		\end{align}
		The extension of $h$ by zero allowed us to replace
		$R-a$ by $\infty$ in the last integral.
		It follows from \eqref{eqMdom} that
		\begin{align}\label{eqAvg}
			\left|
			\int_b^R h(t)T_tf\,dt
			\right|
			&\le
			\int_\tau^{\infty}|T_uJ(h_u)|\,du
			\le
			C\int_\tau^\infty T_uF\,du.
		\end{align}
		
		Since $s(A)<0$, \cite[Proposition 7.1]{DanersGlueckKennedy},
		applied at $\lambda=0$, shows that
		\[
		\int_0^\infty T_uF\,du
		=
		R(0,A)F
		\]
		exists as an improper Riemann integral in $L^p(\Omega)$.
		
		Hence
		\[
		H
		:=
		C\int_\tau^\infty T_uF\,du
		\in L^p(\Omega)_+.
		\]
		Since for each $t \in [b,R]$, $T_tf\ge 0$, \eqref{eqAvg} yields
		\begin{equation}
			\label{eqPavg}
			0
			\le
			\int_b^R h(t)T_tf\,dt
			\le H.
		\end{equation}
		
		Applying \cite[Lemma 1]{Vogt}, we obtain
		\[
		\left(
		\int_b^R|T_tf|^p\,dt
		\right)^{1/p}
		=
		\sup_{\substack{
				0\le h\in L_+^{p'}([b,R])\\
				\norm{h}_{p'}\le 1
		}}
		\int_b^R h(t)T_tf\,dt.
		\]
		Consequently, \eqref{eqPavg} implies
		\[
		\left(
		\int_b^R|T_tf|^p\,dt
		\right)^{1/p}
		\le H.
		\]
		Taking $L^p(\Omega)$-norms gives
		\[
		\int_b^R\norm{T_tf}_p^p\,dt
		\le
		\norm{H}_p^p.
		\]
		Letting $R\to\infty$, we obtain
		\begin{equation}
			\label{eqInteg}
			\int_0^\infty\norm{T_tf}_p^p\,dt
			<\infty
			\qquad
			(f\in L^p(\Omega)_+).
		\end{equation}
		
		Datko's theorem (see \cite[Theorem V.1.8]{EngelNagel}) now implies that $\omega_0(T)<0$.
		
	\end{proof}
	
	\begin{remark}
		By using Corollary \ref{corDom} and a similar argument as in the proof of Theorem \ref{thmMain}, one can also make the proof of $s(A) = \omega_0(T)$ for individually eventually positive semigroups on AM-spaces (see \cite{AroraGlueck2022} for more details) less technical and conceptually much clearer.
		
		Indeed, let $(T_t)_{t\ge 0}$ be an individually eventually positive $C_0$-semigroup on a complex AM-space $E$ with generator $A$. Assume $s(A) < 0$ and fix $f\in E_+$. By individual eventual positivity, $T_t f \ge 0$ for all $t \ge a$ for some $a \ge 0$. Since $E$ is an AM-space, the supremum $g := \sup_{u\in[a,a+1]} T_u f$ exists in $E_+$.
		Set $I := [a,a+1]$ and $Y := \ell^1(I;\mathbb{R})$. The bounded operator $J \colon Y \to E_{\mathbb{R}}$ defined by $Jc := \sum_{u\in I} c_u T_u f$ satisfies $|Jc| \le \norm{c}_1 g$, so $JY \subseteq E_g$.
		Applying Corollary \ref{corDom} to $J$ yields $t_0 \ge a$ and $C \ge 0$ such that, taking canonical unit vectors $c = e_u$, for every $s \ge t_0$ and $u \in [a,a+1]$,
		\[
		|T_s T_u f| \le C T_s g.
		\]
		
		Fix $t \ge t_0 + a + 1$. For $s \in I_t := [t-a-1, t-a]$, we have $s \ge t_0$ and $t-s \in [a,a+1]$, so $|T_t f| = |T_s T_{t-s} f| \le C T_s g$. Integrating over $s \in I_t$ gives
		\[
		|T_t f| \le C \int_{I_t} T_s g \, ds \le C \int_{t_0}^\infty T_s g \, ds =: \widetilde{g}.
		\]
		By \cite[Proposition 7.1]{DanersGlueckKennedy}, the improper Riemann integral $\int_0^\infty T_s g \, ds$ converges in $E$, so $\widetilde{g} \in E_+$. Thus $\sup_{t \ge 0} \norm{T_t f} < \infty$ for all $f \in E_+$. 
		The uniform boundedness principle therefore yields $\sup_{t\ge0}\norm{T_t}<\infty$, so $\omega_0(T)\le0$; applying this to $(e^{-\lambda t}T_t)_{t\ge0}$ for $\lambda>s(A)$ and taking $\lambda\downarrow s(A)$ yields $\omega_0(T)=s(A)$.
	\end{remark}
	
	To conclude, in view of \cite{Arnold2026, ArnoldCoine2023}, it is natural to raise the following question.
	
	\begin{q}
		Let $1 < p < \infty$ and let $(T_t)_{t\ge 0}$ be a Kreiss-bounded, individually eventually positive $C_0$-semigroup on $L^p(\Omega,\mu)$. Does there exist $\varepsilon \in (0,1)$ such that 
		\[
		\norm{T_t} = O(t^{1-\varepsilon}) ?
		\]
	\end{q}
	
	\subsection*{Acknowledgements}
	We are indebted to Jochen Glück for his insightful comments, which significantly contributed to the development of Section 2.

\end{document}